\documentclass[secthm,seceqn,amsthm,ussrhead,reqno, 12pt]{amsart}
\usepackage[utf8]{inputenc}
\usepackage[english]{babel}
\usepackage{amssymb,amsmath,amsthm,amsfonts,xcolor,enumerate,hyperref,comment,longtable,cleveref}
\usepackage[symbol]{footmisc}
\usepackage{times}
\usepackage{cite}
\usepackage{pdflscape}
\usepackage{ulem}
\usepackage[mathcal]{euscript}
\usepackage{hyperref}
\usepackage{cancel}
\usepackage{stmaryrd}

\usepackage{stmaryrd}
\usepackage{xcolor}
\newcommand{\red}{\color{red}}

\theoremstyle{plain}
\newtheorem{theorem}{Theorem}
\newtheorem{thm}[theorem]{Theorem}

\newtheorem{lem}[theorem]{Lemma}

\newtheorem{prop}[theorem]{Proposition}

\newtheorem{rem}[theorem]{Remark}

\newtheorem{defn}[theorem]{Definition}

\newtheorem{cor}[theorem]{Corollary}
\begin{document}

\noindent{\Large 	
Zinbiel superalgebras}\footnote{
The first part of this work is supported by 
 Ministerio de Econom{\'\i}a y Competitividad (Spain), grant PID2020-115155GB-I00 (European FEDER support included, EU);
FCT UIDB/MAT/00212/2020, UIDP/MAT/00212/2020 and by the Spanish Government through the Ministry of Universities grant `Margarita Salas', funded by the European Union - NextGenerationEU. 
The second part of this work is supported by the Russian Science Foundation under grant 22-71-10001.  }

 \

 \bigskip 
 
\begin{center}
    
 {\bf
Luisa Mar\'{i}a Camacho\footnote{Dpto. Matem{\'a}tica Aplicada I, Universidad de Sevilla, Sevilla, Spain;\ lcamacho@us.es  },
Amir Fern\'andez Ouaridi\footnote{ Centro de Matemática, Universidade de Coimbra, Coimbra, Portugal;
University of Cadiz, Puerto Real, Spain; \    amir.fernandez.ouaridi@gmail.com},
Ivan Kaygorodov\footnote{CMA-UBI, Universidade da Beira Interior, Covilh\~{a}, Portugal;
Moscow Center for Fundamental and Applied Mathematics, Moscow,   Russia; \
Saint Petersburg  University, Russia;  
 \    kaygorodov.ivan@gmail.com} \& 
Rosa M. Navarro\footnote{Dpto. de Matem{\'a}ticas, Universidad de Extremadura, C{\'a}ceres, Spain; \ rnavarro@unex.es}

}

\end{center}

	\ 
	
	\noindent {\bf Abstract:} {\it 	
 Throughout the current paper, we extend the study of Zinbiel algebras to Zinbiel superalgebras. In particular, we show that all the Zinbiel superalgebras over an arbitrary field are nilpotent in the same way as occurs for Zinbiel algebras. Moreover and since the most important cases of nilpotent algebras or superalgebras are those with maximal nilpotency index, we study the complex null-filiform Zinbiel superalgebra, i.e. the only one single generated, proving that is unique up to isomorphism. After that, we characterise the naturally graded filiform ones and obtain low-dimensional classifications.
}

 \bigskip


\


  \

	\tableofcontents

\section*{Introduction}

Loday introduced a class of symmetric operads generated by one bilinear operation subject to one relation making each left-normed product of three elements equal to a linear combination of right-normed products:
   $(a_1a_2)a_3=\sum\limits_{\sigma\in \mathbb{S}_3} x_{\sigma} a_{\sigma(1)}(a_{\sigma(2)}a_{\sigma(3)});$
such an operad is called a parametrized one-relation operad. For a particular choice of parameters $\{x_{\sigma}\}$, this operad is said to be regular if each of its components is the regular representation of the symmetric group; equivalently, the corresponding free algebra on a vector space $V$ is, as a graded vector space, isomorphic to the tensor algebra of $V$. 
Bremner and Dotsenko classified, over an algebraically closed field of characteristic zero, all regular parametrized one-relation operads. In fact, they proved that each such operad is isomorphic to one of the following five operads: 
the left-nilpotent operad, the associative operad, the Leibniz operad, the Zinbiel operad, and the Poisson operad \cite{bredo}. Then,  
an algebra $\bf Z$ is called a (left) {\it Zinbiel algebra} if it satisfies the identity 
$(xy)z=x(yz+zy).$  
	Zinbiel algebras were introduced by Loday in \cite{Loday2}.
Under the Koszul duality, the operad of Zinbiel algebras is dual to the operad of Leibniz algebras. 
Zinbiel algebras are also known as pre-commutative algebras \cite{ pasha}
and chronological algebras \cite{Kawski}.
Remark that a Zinbiel algebra is equivalent to a commutative dendriform algebra \cite{comdend}. Also,
the variety of Zinbiel algebras is a proper subvariety in the variety of right commutative algebras and 
each  Zinbiel algebra with the commutator multiplication gives a Tortkara algebra \cite{dzhuma}.
Zinbiel algebras also give an example of  algebras of slowly growing length \cite{GK22}. 
Recently, the notion of matching Zinbiel algebras was introduced in \cite{matchzin} and the defined identities for mono and binary Zinbliel algebras are studied in \cite{ims21}.
Moreover, Zinbiel algebras also  appeared in a study of rack cohomology \cite{rack}, 
number theory \cite{Chapoton21} and 
in  the construction of a Cartesian differential category  \cite{ip21}.
Thus, we can assert that in recent years, there has been a strong interest in the study of Zinbiel algebras in the algebraic and the operad context, see for instance  
\cite{ adashev, classifications, afk22, p-filiform,  SymZinbielSuper, dok,  anau,  matchzin, centr3zinb, dzhuma,  kppv,zinb5,  nilpotency, cam20, CT23,    ims21, identity-dual, 34}. 

Free Zinbiel algebras were shown to
be precisely the shuffle product algebra \cite{34}, 
which is under a certain interest until now \cite{cl}.
Naurazbekova proved that, over a field of characteristic zero, free Zinbiel algebras are the free associative-commutative algebras (without unity) with respect to the symmetrization multiplication and their free generators are found; also she constructed examples of subalgebras of the two-generated free Zinbiel algebra that are free Zinbiel algebras of countable rank \cite{anau}.
Nilpotent algebras play an important role in the class of Zinbiel algebras.
So, 
Dzhumadildaev and  Tulenbaev proved that each complex finite dimensional Zinbiel algebra is nilpotent  \cite{nilpotency};
this result was generalized by Towers for an arbitrary field  \cite{zinb5}.
Naurazbekova 
and Umirbaev proved that in characteristic zero any proper subvariety of the variety of Zinbiel algebras is nilpotent \cite{identity-dual}.
Finite-dimensional Zinbiel algebras with a ``big'' nilpotency index are classified in \cite{adashev,p-filiform}.
Central extensions of three-dimensional Zinbiel algebras were calculated in \cite{centr3zinb} and of filiform Zinbiel algebras in \cite{cam20}.
The description of all degenerations in the variety of complex four-dimensional Zinbiel algebras is given in     \cite{kppv} 
and the geometric classification of complex five-dimensional Zinbiel algebras is given in    \cite{afk22}.
After that, Ceballos and   Towers strudied abelian subalgebras and ideals of maximal dimension in Zinbiel algebras \cite{CT23}.

 Our main goal then, for the present paper, is to  extend the study of Zinbiel algebras to Zinbiel superalgebras. Thus, we prove that all the Zinbiel superalgebras over an arbitrary field are nilpotent analogously as occurs for Zinbiel algebras. Let us remark that the most important cases of nilpotent  superalgebras are those with maximal nilpotency index, therefore we consider the null-filiform and filiform cases. For the former, i.e. the only one single generated, we show that is unique up to isomorphism. For the latter, we characterise the naturally graded for arbitrary dimensions.  Note that among all the gradations, the most important for nilpotent algebras or superalgebras is the natural gradation which comes from the filtration defined by the descending central sequence.
 Finally, we complete the study of Zinbiel superalgebras providing  low-dimensional classifications.

 
\section{Preliminaries and basic definitions}
\subsection{Zinbiel algebras}
We recall first some definitions and basic results regarding Zinbiel algebras.  

\begin{defn}  An algebra ${\bf Z}$  
is called Zinbiel algebra if it satisfies  the identity 
\begin{center}$(x  y)  z= x (y  z+ z  y).$\end{center}
\end{defn} 

For a given Zinbiel (super)algebra ${\bf Z}$,  it is  defined the following sequence:
$${\bf Z}^1={\bf Z},  \quad {\bf Z}^{k+1}={\bf Z}  {\bf Z}^{k}.$$

\begin{defn}   A Zinbiel (super)algebra ${\bf Z}$ is called nilpotent if there exists $s \in \mathbb{N}$ such that ${\bf Z}^s=0$.  The minimal number $s$ satisfying this property is called the nilpotency index of the (super)algebra ${\bf Z}$. 

\end{defn}

 It is not difficult to see that the index of nilpotency of an arbitrary $n$-dimensional nilpotent   Zinbiel  (super)algebra does not exceed the number $n+1$.
Since every finite-dimensional Zinbiel algebra over a field is nilpotent \cite{zinb5}, it made perfect sense to start studying those with the maximal index of nilpotency, i.e. null-filiform. 

 \begin{defn} An $n$-dimensional Zinbiel (super)algebra ${\bf Z}$ is called null-filiform if ${\rm dim} \ {\bf Z}^{i}=n+1-i$. 
 
 \end{defn}
Let us note that a Zinbiel algebra is null-filiform if and only if it is one-generated.
 The classification of complex null-filiform Zinbiel algebras was given in \cite{classifications}.
 
 \begin{thm} An arbitrary $n$-dimensional null-filiform Zinbiel algebra is isomorphic to the following algebra:
 \begin{center}
     $e_i  e_j =C_{i+j-1}^j e_{i+j}.$
      \end{center}
  \end{thm}

After having obtained the aforementioned Zinbiel algebras, the next case for consideration was filiform. Let us denote by $L_x$ the operator of left multiplication on element $x$. 
 Then,  for the operator $L_x$ it has been defined  as a descending sequence $C(x)=(n_1,n_2, \ldots ,n_k)$ with $n_1 + \ldots +n_k =n$, which consists of the dimensions of the Jordan blocks of the operator $L_x$. In the set of such sequences, we consider the lexicographic order, 
 that is, $C(x) = (n_1, n_2, \ldots , n_k) < C(y) = (m_1, m_2, \ldots , m_s)$ if there exists $i$ such that $n_i <m_i$ and $n_j =m_j$ for $j<i$.
 
 \begin{defn}   The sequence $C({\bf Z}) = {\mbox max}\{C(x) : x \in {\bf Z}^1 \backslash {\bf Z}^2\}$ is called the characteristic sequence of the Zinbiel algebra ${\bf Z}$.
 
\end{defn}

\begin{defn} The Zinbiel algebra ${\bf Z}$ is called $p$-filiform if $C({\bf Z})=(n-p,\underbrace{1,\ldots,1}_{p}).$
If $p = 0$ (respectively, $p = 1$), then ${\bf Z}$ is called null-filiform (respectively, filiform) Zinbiel algebra.

\end{defn}

Let ${\bf Z}$ be a finite-dimensional complex Zinbiel algebra with the nilpotency index equal to $s$. 
Let us consider   ${\bf Z}_i={\bf Z}^{i}/{\bf Z}^{i+1}$ and it is denoted by ${\rm gr}({\bf Z})={\bf Z}_1 \oplus {\bf Z}_2 \oplus \ldots \oplus {\bf Z}_{s-1}$. It can be easily checked that ${\rm gr}({\bf Z})$ is a graded Zinbiel algebra. If ${\bf Z}$ and ${\rm gr}({\bf Z})$ are isomorphic, then ${\bf Z}$ is said to be naturally graded. The classification of complex naturally graded filiform Zinbiel algebras was given in  \cite{classifications}:

\begin{thm}  An arbitrary $n$-dimensional ($n \geq 5$) naturally graded complex filiform Zinbiel algebra is isomorphic to the following algebra:

\begin{center}
$e_i  e_j =C_{i+j-1}^j e_{i+j},  \quad \mbox{for \ } \quad 2\leq i+j \leq n-1.$
\end{center}
\end{thm}

Recently in \cite{SymZinbielSuper} there have been studied symmetric (left and right) Zinbiel superalgebras.  Thus,  we extract the following results concerning to (left) Zinbiel superalgebras which we will refer to just as Zinbiel superalgebras.  

\subsection{Zinbiel superalgebras}
The superversion of the   Zinbiel algebras can be obtained in the usual way.

\begin{defn} 
Let ${\bf Z}={\bf Z}_{\bar{0}} \oplus {\bf Z}_{\bar{1}}$ be a ${\mathbb Z}_2$-graded vector space with a bilinear map on ${\bf Z}$ such that 
${\bf Z}_{i}   {\bf Z}_{j} \subset {\bf Z}_{i+j}$.  
${\bf Z}$ is called a  Zinbiel superalgebra if, 
for all homogeneous $x, y, z \in {\bf Z}_{\bar 0} \cup {\bf Z}_{\bar{1}}$ it satisfies 

\begin{center} 
$(x y )z=x\big(yz+ (-1)^{|y| |z|} zy\big).$
\end{center}
As usual,  for  $x \in {\bf Z}_{\bar 0} \cup {\bf Z}_{\bar{1}}$,  
it is defined the corresponding endomorphism of ${\bf Z}$ by 
$L_x(y)=xy$  for all $y \in {\bf Z}_{\bar 0} \cup {\bf Z}_{\bar{1}}$ 
which is called the left multiplication by $x$.  
\end{defn}

\begin{rem}
In the same way that any Zinbiel algebra is a right-commutative algebra, any Zinbiel superalgebra is a right-commutative superalgebra. 
Namely, for any homogeneous $x, y$ and $z$, 
it satisfies the superidentity:
\begin{center}$(x y)z=(-1)^{|y| |z|} (xz)y.$
\end{center}
\end{rem}

Next,  we extend the definitions and first results of Zinbiel algebras to Zinbiel superalgebras. 
Thus,  for a given Zinbiel superalgebra ${\bf Z}={\bf Z}_{\bar{0}} \oplus {\bf Z}_{\bar{1}}$,  we define the following sequence:
$${\bf Z}^1={\bf Z},  \quad {\bf Z}^{k+1}={\bf Z}  {\bf Z}^{k}.$$

Let us note that also they can be defined two sequences as follows: 
$${\bf Z}_{\bar 0}^1={\bf Z}_{\bar 0},  \quad 
{\bf Z}_{\bar 0}^{k+1}={\bf Z}_{\bar 0}  {\bf Z}_{\bar 0}^{k} \quad 
\mbox{ and }  \quad 
{\bf Z}_{\bar 1}^1={\bf Z}_{\bar 1},  \quad 
{\bf Z}_{\bar 1}^{k+1}={\bf Z}_{\bar 0}  {\bf Z}_{\bar 1}^{k}.$$

 Along the last section of the present paper, we will show that all  Zinbiel superalgebras over an arbitrary field are nilpotent,   
 therefore the study of null-filiform and filiform is crucial for the understanding of  finite-dimensional Zinbiel superalgebras.

 Next, we introduce, following the spirit of    the theory of nilpotent superalgebras for Lie and Leibniz cases (see for example \cite{zero-filiform} and reference therein),  the concepts of characteristic sequence and filiform for superalgebras.
  Firstly, let us denote by $L_x$ the operator of left multiplication on a homogeneous even element $x \in {\bf Z}_{\bar 0}$.
 Therefore, we have: 
$L_x :{\bf Z}_{\bar 0} \to {\bf Z}_{\bar 0}$
and we denote by $C_0(x)$ the corresponding descending sequence of the dimensions of Jordan blocks of the operator $L_x$ acting on ${\bf Z}_{\bar 0}$. Analogously and since 
$L_x :{\bf Z}_{\bar 1} \to {\bf Z}_{\bar 1}$
 we denote by $C_1(x)$ the corresponding descending sequence of the dimensions of Jordan blocks of the operator $L_x$ acting on ${\bf Z}_{\bar 1}$.
 Then,  with regard to the lexicographic order we have the following definition.
 
 \begin{defn} The sequence 
 $$
C({\bf Z})=\left(\left.\max_{x\in {\bf Z}_{\bar 0}\backslash {\bf Z}_{\bar 0}^2} C_0(x) \ \right|
\max_{y\in {\bf Z}_{\bar 0}\backslash {\bf Z}_{\bar 0}^2} C_1(y)\right),
$$
 is called the characteristic sequence of the Zinbiel superalgebra ${\bf Z}$.
 
\end{defn}
 Along the present work we assume that both characteristic sequences of the definition are obtained by the same generator element $x\in {\bf Z}_{\bar 0}\backslash {\bf Z}_{\bar 0}^2$ which is usually called characteristic element.
 
\begin{defn}
 A  Zinbiel superalgebra ${\bf Z}={\bf Z}_{\bar{0}} \oplus {\bf Z}_{\bar{1}}$, with ${\rm dim}({\bf Z}_{\bar{0}})=n$ and ${\rm dim}({\bf Z}_{\bar{1}})=m$, is said to be filiform if its characteristic sequence is exactly $C({\bf Z})=(\left. n-1 \right| m ).$
\end{defn} 
 \begin{rem}
 Note that if ${\bf Z}$ is a filiform Zinbiel superalgebra, then ${\bf Z}_{\bar{0}}$ is a filiform Zinbiel algebra.
 \end{rem}
  
  Let us remark that among all the gradations, the most important for nilpotent structures is the natural gradation which comes from the filtration defined by the descending central sequence. Recently, it has been defined the concept of naturally graded for both nilpotent superalgebras, Lie and Leibniz \cite{naturallygraded}. Next we extend this concept for complex Zinbiel superalgebras which are all of them nilpotent.

Consider ${\bf Z}={\bf Z}_{\bar{0}} \oplus {\bf Z}_{\bar{1}}$ to be a complex Zinbiel superalgebra. It can be seen that the sequences aforementioned
$ \{ {\bf Z}_{\bar 0}^k \}$ and $\{{\bf Z}_{\bar 1}^{k} \}$
define a filtration over ${\bf Z}_{\bar 0}$ and ${\bf Z}_{\bar 1}$, respectively. 
If we denote ${\mathfrak z}^{i}_{\bar 0}:={\bf Z}_{\bar 0}^{i-1}/{\bf Z}_{\bar 0}^{i}$ and ${\mathfrak z}^{i}_{\bar 1}:={\bf Z}_{\bar 1}^{i-1}/{\bf Z}_{\bar 1}^{i}$, then it is verified that 
${\mathfrak z}^{i}_{\bar 0} {\mathfrak z}^{j}_{\bar 0} \subset {\mathfrak z}^{i+j}_{\bar 0}$  and  
${\mathfrak z}^{i}_{\bar 0} {\mathfrak z}^{j}_{\bar 1} \subset {\mathfrak z}^{i+j}_{\bar 1}.$

\begin{defn}
Given a complex Zinbiel superalgebra ${\bf Z}={\bf Z}_{\bar{0}} \oplus {\bf Z}_{\bar{1}}$, consider ${\mathfrak z}^{i}={\mathfrak z}^{i}_{\bar 0} \oplus {\mathfrak z}^{i}_{\bar 1}$, with 
${\mathfrak z}^{i}_{\bar 0}={\bf Z}_{\bar 0}^{i-1}/{\bf Z}_{\bar 0}^{i}
 \mbox{  and  }  {\mathfrak z}^{i}_{\bar 1}={\bf Z}_{\bar 1}^{i-1}/{\bf Z}_{\bar 1}^{i}$. 
 Thus, ${\bf Z}$ is said to be  naturally graded if the following conditions hold:
\begin{itemize}
\item[$1.$]  ${\rm gr} ({\bf Z})=\sum_{i \in \mathbb{N}} {\mathfrak z}^{i}$  is a graded  superalgebra (${\mathfrak z}^{i} {\mathfrak z}^{j} \subset {\mathfrak z}^{i+j}$),
\item[$2.$] ${\bf Z}$ and ${\rm gr}({\bf Z})$ are isomorphic.
\end{itemize}
\end{defn}
 
 \section{Null-filiform Zinbiel superalgebras}
\begin{thm} Let ${\bf Z}$ be an $n$-dimensional  null-filiform Zinbiel superalgebra with ${\rm dim}({\bf Z}_{\bar{1}})\neq 0$.  
Then ${\bf Z}$ is isomorphic to  the following superalgebra which occurs only for the cases 
${\rm dim}({\bf Z}_{\bar{0}})={\rm dim}({\bf Z}_{\bar{1}})$ and 
${\rm dim}({\bf Z}_{\bar{1}})={\rm dim}({\bf Z}_{\bar{0}})+1:$
\begin{longtable}{lcllcllcl}
$e_{2k+1}   e_{2l}$&$=$&$C_{k+l}^{k} e_{2k+2l+1},$ \quad &
$e_{2k}   e_{2l}$&$=$&$C_{k+l-1}^{l} e_{2k+2l},$ \quad &
$e_{2k+1}   e_{2l+1}$&$=$&$C_{k+l}^{l} e_{2k+2l+2},$ \end{longtable}
\noindent where $e_{2k},e_{2l} \in {\bf Z}_{\bar{0}}$ and $e_{2k+1},e_{2l+1} \in {\bf Z}_{\bar{1}}.$ 
\end{thm}

\begin{proof} 
It is clear, that each null-filiform Zinbiel superalgebra is one-generated. 
If it is generated by an even element, then it has zero odd part.
Hence,    we can suppose that our superalgebra is generated by an odd element $e_1$ and then, in the same way, as for null-filiform Zinbiel algebras,  we can consider:
$$ e_2:=e_1 e_1, \quad e_3:=e_1(e_1 e_1), \quad \ldots, \quad e_n:= e_1(e_1 \ldots (e_1 (e_1 e_1)) ).$$

Let us remark that the elements above are linearly independent.   This latter fact allows us to regard $\{ e_1, e_2, \dots , e_n\}$ as a basis of the superalgebra ${\bf Z}$  being,   $e_{2k+1}$ odd basis vectors and $e_{2k}$ even ones.  Moreover,  we have only two possibilities: either ${\rm dim}({\bf Z}_{\bar{0}})={\rm dim}({\bf Z}_{\bar{1}})$ or ${\rm dim}({\bf Z}_{\bar{1}})={\rm dim}({\bf Z}_{\bar{0}})+1$.

Let us note that by construction we have 
\begin{equation}
\label{eq1} e_1 e_i=e_{i+1}.
\end{equation}
Now,  we prove by induction 
\begin{equation}
\label{eq2} e_{2k} e_1 =0 \ \mbox{and} \  e_{2k+1} e_1=e_1 e_{2k+1}=e_{2k+2}.
\end{equation}
For $k=1$ the equations hold by considering the following Zinbiel superidentity:
$$e_2 e_1=(e_1 e_1)e_1= e_1(e_1 e_1)-e_1(e_1 e_1)=0$$ and then 
$$e_3 e_1=(e_1 e_2)e_1= e_1(e_2 e_1)+e_1(e_1 e_2)=e_1 e_3$$
Suppose that the equations hold for $k$,  thus 
$$e_{2k+2} e_1=(e_1 e_{2k+1})e_1= e_1(e_{2k+1} e_1)-e_1(e_1 e_{2k+1})=0$$ and then 
$$e_{2k+3} e_1=(e_1 e_{2k+2})e_1= e_1(e_{2k+2} e_1)+e_1(e_1 e_{2k+2})=e_1 e_{2k+3}$$
and therefore we have equations (\ref{eq2}) for $k+1$. Following we prove, also by induction,  the equations 
\begin{equation}
\label{eq3} e_{2k} e_2 =k e_{2k+2}, \ e_{2k+1} e_2=(k+1)e_{2k+3},  \
e_2 e_{2k}=e_{2k+2}, \ e_2 e_{2k+1}=0.
\end{equation}

 The equations hold for $k=1$ by considering the following Zinbiel superidentities:
\begin{longtable}{lclclclcl}
$e_2 e_2$&$=$&$(e_1 e_1)e_2 $&$=$&$ e_1(e_1 e_2)+e_1(e_2e_1)$&$=$&$e_1e_3$&$=$&$e_4,$\\  
$e_3 e_2$&$=$&$(e_1 e_2)e_2$&$=$&$ e_1(e_2 e_2)+e_1(e_2 e_2)$&$=$&$2e_5,$\\
$0$&$=$&$(e_2 e_1)e_2$&$=$&$e_2(e_1 e_2)+e_2(e_2 e_1)$&$=$&$e_2e_3,$
\end{longtable}
supposing that the equations hold for $k$ we get the equations for $k+1$:
\begin{longtable}{rclclclclclcl}
$e_{2(k+1)} e_2$&$=$&$(e_1 e_{2k+1})e_2$&$=$&$ e_1(e_{2k+1} e_2)+e_1(e_2e_{2k+1})$&$=$&$(k+1)e_{2(k+1)+2},$\\
$e_2e_{2(k+1)} $&$=$&$(e_1 e_1)e_{2k+2} $&$=$&$ e_1(e_1 e_{2k+2} )+e_1(e_{2k+2}e_{1})$&$=$&$e_{2(k+1)+2},$\\ 
$e_{2(k+1)+1} e_2$&$=$&$(e_1 e_{2k+2})e_2$&$=$&$ e_1(e_{2k+2} e_2)+e_1(e_2e_{2k+2})$&$=$&$((k+1)+1)e_{2(k+1)+3},$\\ 
$0$&$=$&$(e_2 e_1)e_{2k+2}$&$=$&$e_2(e_1 e_{2k+2})+e_2(e_{2k+2} e_1)$&$=$&$e_2e_{2(k+1)+1}.$
\end{longtable}
Next,  on account of equations (\ref{eq2}) together with the fact that $e_{2k}e_{2l}$ is a multiple of $e_{2k+2l}$ we have 
\begin{center}$0=(e_{2k} e_{2l})e_1= e_{2k}(e_{2l} e_1)+e_{2k}(e_1 e_{2l})=e_{2k}e_{2l+1},$
\end{center}
which leads to the equation
 $e_{2k}  e_{2l+1}=0.$
Following we prove by induction the equation 
\begin{equation}
\label{eq5} e_{2k+1}  e_{2l}=C_{k+l}^{k} e_{2k+2l+1}.
\end{equation}
From the equality (\ref{eq1}) we get equation (\ref{eq5}) for $k=0$ and every $l$.  Now after supposing that equation (\ref{eq5}) holds for $k$ and every $l$,  we obtain that it also holds for $k+1$ and every $l$ taking into account equality (\ref{eq3}):
\begin{center}
    $e_{2(k+1)+1} e_{2l}=\frac{1}{k+1}(e_{2k+1} e_2)e_{2l}=
\frac{1}{k+1}\big(e_{2k+1} (e_2e_{2l})+e_{2k+1}(e_{2l}e_2)\big)=\frac{1+l}{k+1}e_{2k+1}e_{2l+2},$
\end{center}
but since    
$e_{2k+1}e_{2l+2}= C_{k+l+1}^{k} e_{2k+2l+3}$,   then 
\begin{center}$\frac{1+l}{k+1} C_{k+l+1}^{k} =\frac{(1+l)(k+l+1)!}{(k+1)k! (l+1)!} =\frac{(k+l+1)!}{(k+1)! l!}=C_{(k+1)+l}^{k+1},$\end{center}
which concludes the proof of equality (\ref{eq5}).  Similarly, we prove the equation 
\begin{equation}
\label{eq6} e_{2k+1}  e_{2l+1}=C_{k+l}^{l} e_{2k+2l+2}.
\end{equation}
 The equality (\ref{eq1}) leads to equation (\ref{eq6}) for $k=0$ and every $l$,  and  after supposing that equation (\ref{eq6}) holds for $k$ and every $l$,  we obtain that it also holds for $k+1$ and every $l$ on account of equality (\ref{eq3}):

\begin{center}
    $e_{2(k+1)+1} e_{2l+1}=\frac{1}{k+1}(e_{2k+1} e_2)e_{2l+1}=
\frac{1}{k+1} \big(e_{2k+1} (e_2e_{2l+1})+e_{2k+1}(e_{2l+1}e_2)\big)=\frac{1+l}{k+1}e_{2k+1}e_{2l+3},$
\end{center}

but since    $e_{2k+1}e_{2l+3}= C_{k+l+1}^{l+1} e_{2k+2l+4}$,   then 
\begin{center}$\frac{1+l}{k+1} C_{k+l+1}^{l+1} =\frac{(1+l)(k+l+1)!}{(k+1)k! (l+1)!} =\frac{(k+l+1)!}{(k+1)! l!}=C_{(k+1)+l}^{l},$\end{center}
which concludes the proof of equality (\ref{eq6}).  Finally,  we prove the last equality 
\begin{equation}
\label{eq7} e_{2k}  e_{2l}=C_{k+l-1}^{l} e_{2k+2l},
\quad k \geq 1,  \ l  \geq 1.
\end{equation}
 The equality (\ref{eq3}) leads to equation (\ref{eq7}) for $k=1$ and every $l$,  and  after supposing that equation (\ref{eq7}) holds for $k$ and every $l$,  we obtain that it also holds for $k+1$ and every $l$ on account of equality (\ref{eq3}):
\begin{center}$e_{2(k+1)} e_{2l}=\frac{1}{k}(e_{2k} e_2)e_{2l}=
\frac{1}{k}(e_{2k} \big(e_2e_{2l})+e_{2k}(e_{2l}e_2)\big)=\frac{1+l}{k}e_{2k}e_{2l+2}$
\end{center}
but as    $e_{2k}e_{2l+2}= C_{k+l}^{l+1} e_{2k+2l+2}$,   then
\begin{center}
$\frac{1+l}{k} C_{k+l}^{l+1} =\frac{(1+l)(k+l)!}{(k-1)!k (l+1)!} =\frac{(k+l)!}{k! l!}=C_{(k+1)+l-1}^{l},$\end{center}
which concludes the proof of equality (\ref{eq7}) and also of the statement of the Theorem.

\end{proof}

\section{Naturally graded filiform Zinbiel superalgebras}

\begin{lem} Let  ${\bf Z}$  be a complex naturally graded filiform Zinbiel superalgebra with ${\rm dim}({\bf Z}_{\bar{0}})=n$, $n\geq 5$  and ${\rm dim}({\bf Z}_{\bar{1}})=m$. Then, there are  a basis 
 $\{e_1, \dots,e_n\}$ for  ${\bf Z}_{\bar{0}}$  and a basis $\{f_1, \dots,f_m\}$ for  ${\bf Z}_{\bar{1}}$, 
 for which we have the following multiplication table
$$\begin{array}{ll}
e_i  e_j =C_{i+j-1}^j e_{i+j},  &    2\leq i+j \leq n-1, \\
e_1f_j=f_{j+1}, & 1\leq j \leq m-1.
\end{array}$$
 \end{lem}

\begin{proof}
  Consider ${\bf Z}={\bf Z}_{\bar{0}}\oplus {\bf Z}_{\bar{1}}$ to be a naturally graded filiform Zinbiel superalgebra with ${\rm dim}({\bf Z}_{\bar{0}})=n$ and ${\rm dim}({\bf Z}_{\bar{1}})=m$. We set $\{e_1, \dots,e_n\}$  and $\{f_1, \dots,f_m\}$ as bases of ${\bf Z}_{\bar{0}}$ and ${\bf Z}_{\bar{1}}$ respectively. It derives from the definition that ${\bf Z}_{\bar{0}}$ is a naturally graded Zinbiel algebra and then, for $n\geq 5$ we have from \cite{classifications} completely determined its products
$ e_i  e_j =C_{i+j-1}^j e_{i+j},$   for $2\leq i+j \leq n-1.$
Since $e_1$ is the characteristic element above then there is no loss of generality in supposing  $e_1f_j=f_{j+1}, 1\leq j \leq m-1.$
\end{proof}

Let us use the following notion $\prod\limits_{k=0}^{-1}f(k):=1.$

\begin{thm}\label{general} 
Let  ${\bf Z}$  be a complex naturally graded filiform Zinbiel superalgebra with ${\rm dim}({\bf Z}_{\bar{0}})=n$, $(n\geq 5),$  and 
${\rm dim}({\bf Z}_{\bar{1}})=m,$ $(m>3)$.  
Then ${\bf Z}$ is isomorphic to one of the following superalgebras:

\begin{longtable}{rcl}
    
$
\mathfrak{nf}_1$&$:$&$\left\{\begin{array}{lll}
e_i  e_j =C_{i+j-1}^j e_{i+j},  & & \begin{array}{l} 2\leq i+j \leq n-1.\end{array} \\[1mm]


e_if_j=\frac{\prod\limits_{k=0}^{i-2}(j+k-1) }{(i-1)!}f_{j+i}, &

f_je_i=\frac{\prod\limits_{k=0}^{i-1}(j+k-2) }{i!}f_{j+i}, &
\begin{array}{l}
1\leq i\leq n-1.
\end{array}\\

e_nf_1= f_{2},&f_1e_n=-f_2.& \\
\end{array}\right.$\\

\\
$\mathfrak{nf}_2^\alpha$&$:$&$\left\{\begin{array}{lll}
e_i  e_j =C_{i+j-1}^j e_{i+j},  &   & \begin{array}{l} 2\leq i+j \leq n-1. \end{array}\\[1mm]
e_if_j=\frac{\prod\limits_{k=0}^{i-2}(\alpha+j+k) }{(i-1)!}f_{j+i}, & 
f_je_i=\frac{\prod\limits_{k=0}^{i-1}(\alpha+j+k-1) }{i!}f_{j+i},  &\begin{array}{l}
1\leq i\leq n-1.
\end{array}\\
\end{array}\right.$\\

\\
$
\mathfrak{nf}_3$&$:$&$\left\{
\begin{array}{lll}
e_i  e_j =C_{i+j-1}^j e_{i+j},  &   & \begin{array}{l}2\leq i+j \leq n-1. \end{array}\\[1mm]

e_if_j=\frac{\prod\limits_{k=0}^{i-2}(2-m+j+k)}{(i-1)!} f_{j+i}, & 
f_je_i=\frac{\prod\limits_{k=0}^{i-1}(1-m+j+k) }{i!}f_{j+i}, &  
\begin{array}{l}
1\leq i\leq n-1,
\end{array}\\[1mm]
e_nf_{m-1}=f_m.& &

\end{array}\right.
 $
 \end{longtable}
 
If $m\geq n-2$, the superalgebra ${\bf Z}$ is either isomorphic to  one of the previous superalgebras or isomorphic  to one of the next two:
\begin{longtable}{rcl}
$
\mathfrak{nf}_4$&$:$&$\left\{\begin{array}{lll}
e_i  e_j =C_{i+j-1}^j e_{i+j},  &   &  \begin{array}{l}2\leq i+j \leq n-1. \end{array}\\[1mm]


e_if_j=\frac{\prod\limits_{k=0}^{i-2}(n-3+j+k)}{(i-1)!} f_{j+i},&
f_je_i=\frac{\prod\limits_{k=0}^{i-1}(n-4+j+k)}{i!} f_{j+i}, & 
\begin{array}{l}
1\leq i\leq n-1.\end{array}\\[1mm]

f_1f_{n-2}= e_{n-1}.&&
\end{array}\right.$\\

\\
$
\mathfrak{nf}_5$&$:$&$
\left\{\begin{array}{lll}
e_i  e_j =C_{i+j-1}^j e_{i+j},  &  &   \begin{array}{l}2\leq i+j \leq n-1. \end{array}\\[1mm]

e_if_j=\frac{\prod\limits_{k=0}^{i-2}(3-n+j+k) }{(i-1)!}f_{j+i},& f_je_i=\frac{\prod\limits_{k=0}^{i-1}(2-n+j+k)}{i!} f_{j+i}, & \begin{array}{l}
1\leq i\leq n-1.\end{array}\\

e_nf_{n-2}=f_{n-1},&
f_1f_{n-2}= e_{n-1}.&
\end{array}\right.
$
\end{longtable}
\end{thm}

\begin{proof}
We consider now, ${\bf Z}={\bf Z}_{\bar{0}}\oplus {\bf Z}_{\bar{1}}$   a complex naturally graded filiform Zinbiel superalgebra with ${\rm dim}({\bf Z}_{\bar{0}})=n$, $n\geq 5$  and 
${\rm dim}({\bf Z}_{\bar{1}})=m$ with $m\geq 3.$ 
Then, there exists a basis $\{e_1, \dots,e_n, f_1, \ldots ,f_m \}$ and we have one of the following possibilities depending on if $m<n-1$, $m=n-1$ or $m>n-1$.
\begin{longtable}{cccccccccccccc}
$\underbrace{ \langle e_1,e_n,f_1 \rangle}_{ {{\mathfrak z}^1} }$ & $\oplus$ & 
$\underbrace{\langle e_2, f_2\rangle}_{ { {\mathfrak z}^2 }}$ & 
$\oplus$ & $\ldots$ & $\oplus$ & 
$\underbrace{ \langle e_m , f_m \rangle}_{ { {\mathfrak z}^m }}$& $\oplus$ & 
$\underbrace{ \langle e_{m+1}\rangle}_{ { {\mathfrak z}^{m+1} }}$ & $\oplus$ & 
$\ldots$ & $\oplus$ & 
$\underbrace{ \langle e_{n-1}\rangle}_{ { {\mathfrak z}^{n-1} }}$ \\

$\underbrace{\langle e_1,e_n,f_1 \rangle}_{ { {\mathfrak z}^1 }}$ & $\oplus$ & $\underbrace{\langle e_2, f_2\rangle}_{ { {\mathfrak z}^2 }}$ & $\oplus$ & 
$\ldots$ & $\oplus$ & 
$\underbrace{ \langle e_{n-1} , f_{n-1}\rangle}_{ { {\mathfrak z}^{n-1} }}$ \\

$\underbrace{\langle e_1,e_n,f_1 \rangle}_{ {\mathfrak z}^1 }$ & $\oplus$ & $\underbrace{\langle e_2, f_2\rangle}_{ { {\mathfrak z}^2 }}$ & 
$\oplus$ & $\ldots$ & $\oplus$ & $\underbrace{ \langle e_{n-1} , f_{n-1}\rangle}_{ { {\mathfrak z}^{n-1} }}$& $\oplus$ & 
$\underbrace{ \langle f_{n} \rangle}_{ { {\mathfrak z}^n }}$ & 
$\oplus$ & $\ldots$ & $\oplus$ & 
$\underbrace{ \langle f_m\rangle}_{ { {\mathfrak z}^m }}$ \\
\end{longtable} 


We will study  these three cases together. We consider 
$f_1e_1=\alpha f_2$. By induction using the Zinbiel superidentity we obtain:
\begin{longtable}{lcllcll}

$e_if_j$&$=$&$\frac{\prod\limits_{k=0}^{i-2}(\alpha+j+k) }{(i-1)!}f_{j+i},$& 
$f_je_i$&$=$&$\frac{\prod\limits_{k=0}^{i-1}(\alpha+j+k-1) }{i!}f_{j+i},$ & 
$\begin{array}{l}
1\leq i\leq n-1,
\end{array}$\\[2mm]

$f_1 e_n$&$=$&$c_1f_2.$&
$e_nf_j$&$=$&$b_j f_{j+1},$& $\begin{array}{l}1\leq j\leq m-1.\end{array}$\\

\end{longtable}

For every three homogeneous elements $a,b,c \in {\bf Z}$, we define
\begin{center}
    $\mathfrak{sZ}\{a,b,c\}:=(ab)c-a\big(bc+(-1)^{|a||b|}cb\big).$
    \end{center}
Now, applying the  Zinbiel superidentity for basis elements gives the following relations.

\begin{longtable}{llllll}
$\Big\{\mathfrak{sZ}\{e_n,e_1,f_j\}=0\Big\}_{2\leq j\leq m-1}$& $\Rightarrow$ & $(\alpha+j-1)b_j=0$&\\

$\Big\{\mathfrak{sZ}\{e_n,f_j,e_1\}=0\Big\}_{1\leq j\leq m-2}$&$\Rightarrow$ &    $(\alpha+j)b_j=0$& $\Rightarrow \fbox{$b_j=0, \ 2\leq j\leq m-2$}$ \\

$\Big\{\mathfrak{sZ}\{f_1,e_n,e_1\}=0\Big\}$&$\Rightarrow$ & $(\alpha+1)c_1=0,$&\\
$\Big\{\mathfrak{sZ}\{f_2,e_n,e_1\}=0\Big\}$&$\Rightarrow$ & $(\alpha+2)(c_1+b_1)=0,$&$\Rightarrow \fbox{$c_1=-b_1$}$\\
\end{longtable}
thus, $f_1e_n=-b_1f_2$ and $f_je_n=0$ with $2\leq j\leq m-1.$
Then, we have the following products:
\begin{longtable}{lcllcll}
$e_i  e_j $&$=$&$C_{i+j-1}^j e_{i+j},$  && &&  $\begin{array}{l}2\leq i+j \leq n-1,\end{array}$ \\


$e_if_j$&$=$&$\frac{\prod\limits_{k=0}^{i-2}(\alpha+j+k)}{(i-1)!} f_{j+i},$ &
$f_je_i$&$=$&$\frac{\prod
\limits_{k=0}^{i-1}(\alpha+j+k-1)}{i!} f_{j+i},$ &  $\begin{array}{l}
1\leq i\leq n-1,\end{array}$\\

$e_nf_1$&$=$&$b_1 f_{2},$& $e_nf_{m-1}$&$=$&$b_{m-1}f_m,$&\\
$f_1e_n$&$=$&$-b_1f_2,$&$f_je_n$&$=$&$0,$&$\begin{array}{l}2\leq j\leq m-1,\end{array}$
\end{longtable}
with $(\alpha+1)b_1=(\alpha+m-2)b_{m-1}=0.$

Finally, only rest to compute the products $f_if_j=h_{ij}e_{i+j}.$ We compute them by induction:

\begin{itemize}
\item {\bf Step 1:} We have:
\begin{longtable}{lcl}
$f_1f_1$&$=$&$h_{11}e_2,$\\
$f_2f_1$&$=$&$(e_1f_1)f_1=e_1(f_1f_1)-e_1(f_1f_1)=0,$\\
$f_1f_2$&$=$&$h_{12}e_3.$
\end{longtable}

Moreover, we have 
\begin{longtable}{lclcl}
$\mathfrak{sZ}\{f_1,f_1,e_1\}=0$ & $\Rightarrow$&$2h_{11}=(a+1)h_{12},$&\\
$\mathfrak{sZ}\{f_1,e_1,f_1\}=0$&$\Rightarrow$&$(\alpha+1)h_{12}=0$&$\Rightarrow$ &$ h_{11}=(a+1)h_{12}=0.$\\
\end{longtable}

\item {\bf Step 2:} Now, we can write:
\begin{longtable}{lc rccclcl}
&&$f_1f_1$&$=$&$f_2, \quad \quad  \quad \quad \quad \quad f_2 f_1$&$=$&$0,$\\
&&$f_1f_2$&$=$&$h_{12}e_3,\quad \quad \quad \quad f_1f_3$&$=$&$h_{13}e_4,$\\
$f_2f_2$&$=$&$(e_1f_1)f_2$&$=$&$e_1(f_1f_2)-e_1(f_2f_1)$&$=$&$h_{12}(e_1e_3)$&$=$&$h_{12}e_4,$\\
$f_3f_1$&$=$&$(e_1f_2)f_1$&$=$&$e_1(f_2f_1)-e_1(f_1f_2)$&$=$&$-h_{12}(e_1e_3)$&$=$&$-h_{12}e_4,$
\end{longtable}
with $(\alpha+1)h_{12}=0.$
Applying the Zinbiel superidentity, we get:
\begin{longtable}{llll}
$\mathfrak{sZ}\{f_1,f_2,e_1\}=0$&$\Rightarrow$& $3h_{12}=(\alpha+2)h_{13},$\\

$\mathfrak{sZ}\{f_1,e_1,f_2\}=0$&$\Rightarrow$& $(\alpha+2)h_{13}=\alpha h_{12}.$
\end{longtable}

Therefore, joining the following equations:
$$\left.\begin{array}{rcl}
(\alpha+1)h_{12} & = & 0\\
3h_{12}& = &(\alpha+2)h_{13}\\
(\alpha+2)h_{13} & =& \alpha h_{12}\\
\end{array}\right \}\Rightarrow h_{12}=0,\ (\alpha+2)h_{13}=0.$$

\item {\bf Step 3:} We suppose that 
\begin{longtable}{rclclclcl}
$f_if_{k+1-i}$&$=$&$0,$&& $2\leq i\leq k,$\\
$f_1f_k$&$=$&$h_{1k}e_{k+1},$\\
$f_1f_{k+1}$&$=$&$h_{1\, k+1} e_{k+2},$\\

$f_{k+1}f_1$&$=$&$(e_1f_k)f_1$&$=$&$e_1(f_kf_1)-e_1(f_1f_k)=-h_{1k}(e_1e_{k+1})=-h_{1k}e_{k+2},$\\
$f_2f_k$&$=$&$(e_1f_1)f_k$&$=$&$e_1(f_1f_k)-e_1(f_kf_1)=h_{1k}(e_1e_{k+1})=h_{1k} e_{k+2},$\\
$f_kf_2$&$=$&$(e_1f_{k-1})f_2$&$=$&$e_1(f_{k-1}f_2)-e_1(f_2f_{k-1})=0,$\\
$f_if_{k+2-i}$&$=$&$(e_1f_{i-1})f_{k+2-i}$&$=$&$e_1(f_{i-1}f_{k+2-i})-e_1(f_{k+2-i}f_{i-1})=0,$\\
\multicolumn{7}{l}{(by induction hypothesis).}
\end{longtable}
with $(\alpha+k-1)h_{1k}=0.$

Now, we apply the Zinbiel superidentity obtaining:
\begin{longtable}{lllcl}
$\mathfrak{sZ}\{f_1,f_k,e_1\}=0$&$\Rightarrow$& $(k+1)h_{1k}$&$=$&$(\alpha+k)h_{1\, k+1},$\\
$\mathfrak{sZ}\{f_1,e_1,f_k\}=0$&$\Rightarrow$& $(\alpha+k)h_{1\, k+1}$&$=$&$a h_{1k}.$
\end{longtable}

Therefore, joining the following equations:
$$\left.\begin{array}{rcl}
(\alpha+k-1)h_{1k}&=&0\\
(k+1)h_{1k}&=&(\alpha+k)h_{1\, k+1}\\
(\alpha+k)h_{1\, k+1}&=&\alpha h_{1k}\\
\end{array}\right \}\Rightarrow h_{1k}=0,\ (\alpha+k)h_{1\, k+1}=0.$$
\end{itemize}

This process is finite, then we get  
$f_1f_m=h_{1m}e_{m+1}$ (the rest $f_{i}f_j=0$) with  $(\alpha+m-1)h_{1m}=0$ if $m<n-2$ and $f_1f_{n-2}=h_{1\, n-2}e_{n-1}$ (the rest $f_{i}f_j=0$) with  $(\alpha+n-3)h_{1\, n-2}=0$ if $m\geq n-2.$ 

In the case $m<n-2$: $\mathfrak{sZ}\{f_1,f_m,e_1\}=0$ gives    $h_{1m}=0.$

Thus, the Zinbiel superalgebra has the following multiplication table:
\begin{longtable}{lcllcll}
$e_i  e_j $&$=$&$C_{i+j-1}^j e_{i+j},$  &&& &    $\begin{array}{l}2\leq i+j \leq n-1,\end{array}$ \\[1mm]


$e_if_j$&$=$&$\frac{\prod\limits_{k=0}^{i-2}(\alpha+j+k)}{(i-1)!} f_{j+i},$ & 
$f_je_i$&$=$&$\frac{\prod\limits_{k=0}^{i-1}(\alpha+j+k-1)}{i!} f_{j+i},$ & 
$\begin{array}{l}
1\leq i\leq n-1,\end{array}$\\

$e_nf_1$&$=$&$b_1 f_{2},$&
$e_nf_{m-1}$&$=$&$b_{m-1}f_m,$& $\begin{array}{l} m\geq n-2\end{array}$\\
$f_1e_n$&$=$&$-b_1f_2,$ &$ f_1f_{n-2}$&$=$&$h e_{n-1},$
\end{longtable}
with $(\alpha+1)b_1=(\alpha+m-2)b_{m-1}=0$ and $(\alpha+n-3)h=0$ if $m\geq n-2.$ 
Now, we consider $m>3$ and $m<n-2.$ We can distinguish the following cases:

\medskip 

\begin{enumerate}

    \item 
[{\bf Case 1:}] $b_1\neq 0.$ In this case, we have $\alpha=-1$ and $b_{m-1}=0.$ If $m\geq n-2,$ we also have that $h=0.$ Then we obtain the Zinbiel superalgebra $\mathfrak{nf}_1.$

\item 
[{\bf Case 2:}] $b_1= 0$ and $b_{m-1}=0.$ In this case, we have a family of Zinbiel superalgebras $\mathfrak{nf}_2^\alpha.$

\item 
[{\bf Case 3:}] $b_1= 0$ and $b_{m-1}\neq 0.$ Then, $\alpha=2-m.$ We have the superalgebra $\mathfrak{nf}_3.$

\end{enumerate}

\medskip

Now, we consider $m>3$ and $m\geq n-2.$ Analyzing the above cases, we have a new superalgebra in  {\bf Case 2} (when $\alpha=n-3$ and $h\neq 0$, we can consider $h=1$), $\mathfrak{nf}_4,$ and another in {\bf Case 3} (when  $m=n-1,$ $\alpha=3-n,$ and $h\neq 0$, we can consider $h=1$), $\mathfrak{nf}_5$. 
\end{proof}

Next, we study the particular case $m=3.$ We consider now, ${\bf Z}$   a complex naturally graded filiform Zinbiel superalgebra with ${\rm dim}({\bf Z}_{\bar{0}})=n$, $n\geq 5$ and ${\rm dim}({\bf Z}_{\bar{1}})=3$. 
Then, there exists a basis $\{e_1, \dots,e_n, f_1, f_2,f_3 \}$ and we have the gradation:
$$\begin{array}{cccccccccccc}
\underbrace{\langle e_1,e_n,f_1\rangle } & \oplus & \underbrace{\langle e_2, f_2\rangle } & \oplus &\underbrace{ \langle e_3,f_3\rangle } & \oplus & \underbrace{ \langle e_4\rangle}& \oplus & \underbrace{ \langle e_5\rangle} & \oplus &  \dots  \\
{\mathfrak z}^1 &  &  {\mathfrak z}^2 & & {\mathfrak z}^3 & & {\mathfrak z}^4 & & {\mathfrak z}^5 
\end{array}$$

\begin{thm} Let  ${\bf Z}$  be a complex naturally graded filiform Zinbiel superalgebra with ${\rm dim}({\bf Z}_{\bar{0}})=n$, $(n\geq 5),$  and ${\rm dim}({\bf Z}_{\bar{1}})=3.$  Then ${\bf Z}$ is isomorphic to  one of the following superalgebras:
\begin{itemize}
\item If $n>5,$ then ${\bf Z}$ is  isomorphic either to $\mathfrak{nf}_1,$   $\mathfrak{nf}_2^\alpha,$   $\mathfrak{nf}_3,$ or 
 $${\mathfrak a}_1:\left\{\begin{array}{lll}
e_i  e_j =C_{i+j-1}^j e_{i+j},  &    2\leq i+j \leq n-1, &\\
e_1 f_1=f_{2},& e_1 f_2=f_{3},&f_1 e_1=-f_{2},\\[1mm]
e_nf_1= f_{2},&
e_nf_{2}=f_3,&
f_1e_n=-f_2.
\end{array}\right.$$

\item If $n=5, $ then ${\bf Z}$ is isomorphic either to $\mathfrak{nf}_1,$  $\mathfrak{nf}_2^{\alpha\neq -2},$   $\mathfrak{nf}_3,$ $\mathfrak{a}_1$ or 
$$
\mathfrak{a}_2:\left\{\begin{array}{llll}
e_i  e_j =C_{i+j-1}^j e_{i+j},  &   2\leq i+j \leq n-1, &\\
f_1 e_1=-2f_{2},&f_2 e_1=-f_{3},& f_1e_2= f_{3}, \\
e_1 f_1=f_{2},& e_1 f_2=f_{3}, & 
e_2f_1=- f_{2}, &f_1f_{n-2}= e_{n-1}. 
\end{array}\right.$$
\end{itemize}
\end{thm}

\begin{proof}
Similar to general case (see the proof of Theorem \ref{general}) we get  that ${\bf Z}$ is isomorphic to 
\begin{longtable}{lcllcllcllcl}
$e_i  e_j$&$ =$&$C_{i+j-1}^j e_{i+j},$  &   \multicolumn{3}{l}{$2\leq i+j \leq n-1,$} &\\

$f_1 e_1$&$=$&$af_{2},$&
$f_1e_2$&$=$&$\frac{\alpha(\alpha+1)}{2} f_{3},$ &  
$f_2 e_1$&$=$&$(\alpha+1)f_{3},$& 
$e_nf_1$&$=$&$b_1 f_{2},$\\

$e_1 f_1$&$=$&$f_2,$&
$e_2f_1$&$=$&$(\alpha+1) f_{2},$ &
$e_1 f_2$&$=$&$f_3,$ &
 $e_nf_{2}$&$=$&$b_{2}f_3,$\\

$f_1e_n$&$=$&$-b_1f_2,$&
\multicolumn{5}{l}{$(f_1f_{3}=h e_{4},\ \mbox{if } \ n=5.),$}
\end{longtable}
with $(\alpha+1)b_1=(\alpha+1)b_{2}=0$ and $(\alpha+2)h=0$ if $n=5.$

First, we consider $n>5.$ We can distinguish the following cases:
\begin{itemize}
\item $b_1\neq 0$ and $b_2\neq 0.$ In this case, $\alpha=-1$ and by a change of basis we can consider $b_1=b_2=1$ obtaining the superalgebra $\mathfrak{a}_1.$
\item $b_1\neq 0$ and $b_2=0.$ This case is similar to general case and we obtain $\mathfrak{nf}_1.$
\item $b_1=0.$ This case  is similar to general case and we obtain $\mathfrak{nf}_3$ (for $b_2\neq 0$) and $\mathfrak{nf}_2^\alpha$ (for $b_2=0$)
\end{itemize}

Now, we consider $n=5$, we have a new product and a new restriction. A similar study  allows getting  the superalgebras of the statement. 
\end{proof}

\section{Classification of low-dimensional complex Zinbiel superalgebras}


\begin{lem}
\label{lem(n-1,1)}
Given a $n$-dimensional Zinbiel superalgebra ${\bf Z}$ of type $(n-1, 1)$,  i.e. ${\rm dim}({\bf Z}_{\bar 0})=n-1$ and ${\rm dim}({\bf Z}_{\bar 1})=1$,   then ${\bf Z}_{\bar 0}{\bf Z}_{\bar 1}={\bf Z}_{\bar 1}{\bf Z}_{\bar 0}=\left\{0\right\}$. Moreover, ${\bf Z}_{\bar 1}{\bf Z}_{\bar 1}$ is a subspace of ${\rm Ann}_{\rm L}\,({\bf Z}_{\bar 0})$. 
\end{lem}
\begin{proof}
Let $e_1, \ldots, e_{n-1}$ be a base of ${\bf Z}_{\bar 0}$ and $f_1$ a base of ${\bf Z}_{\bar 1}$.
Denote by $e_{i}f_{1} = a_{i}f_{1}$ and $f_{1}e_i = b_{i}f_{1}$, for $a_{i}, b_{i} \in \mathbb{C}$ and  $i=1,\ldots,n-1$. On the one hand, for $i=1,\ldots,n-1$ we have
$$(e_{i}f_{1}) e_{i} = e_{i}(f_{1}e_{i}) + e_{i}(e_{i} f_{1}); \quad a_{i}b_{i}f_{1} = a_{i}b_{i}f_{1} + a_{i}^2 f_{1}.$$
Hence, we have  $a_{i}=0$ and ${\bf Z}_{\bar 0}{\bf Z}_{\bar 1} = \left\{0\right\}$. 

On the other hand,  since ${\bf Z}_{\bar 0}$ is a Zinbiel algebra, it is nilpotent. Suppose ${\bf Z}_{\bar 0}^{s}=0$ and proceed by induction. For $x\in {\bf Z}_{\bar 0}^{s-1}$, we have 
\begin{center}$(f_{1}x)x = f_{1}(xx+xx)=2f_{1}(xx),$\end{center}
 and since $xx\in {\bf Z}_{\bar 0}^{s}=\left\{0\right\}$, we obtain $(f_{1}x)x=0$, which implies $f_{1}x=0$, and we may write $f_{1}{\bf Z}_{\bar 0}^{s-1}=\left\{0\right\}$. Now, suppose  $f_{1}{\bf Z}_{\bar 0}^{s-k+1} =\left\{0\right\}$, then for $x\in {\bf Z}_{\bar 0}^{s-k}$, we have $(f_{1}x)x=2f_{1}(xx) = 0$, because $xx\in {\bf Z}_{\bar 0}^{s-k+1}$. Therefore, $f_{1}x=0$ and we have $f_{1}{\bf Z}_{\bar 0}^{s-k} = \left\{0\right\}$. For $k=s-1$, we conclude that ${\bf Z}_{\bar 1}{\bf Z}_{\bar 0}^{1} = {\bf Z}_{\bar 1}{\bf Z}_{\bar 0} = \left\{0\right\}$.

\bigskip 

Further, we have
$(f_{1}f_{1})x = f_{1}(f_{1}x) + f_{1}(xf_{1}) = 0$ for $x\in {\bf Z}_{\bar 0}$.  Therefore, $(f_{1} f_{1}) \in {\rm Ann}_{\rm  L}\,({\bf Z}_{\bar 0})$ and ${\bf Z}_{\bar 1}{\bf Z}_{\bar 1}$ is a subspace of ${\rm Ann}_{\rm  L}\,({\bf Z}_{\bar 0})$ being  ${\rm Ann}_{\rm  L}\,({\bf Z}_{\bar 0}) = \left\{ x\in A: x{\bf Z}_{\bar 0}=0\right\}$.
\end{proof}

The converse is a straightforward verification.

\begin{rem} \label{remconverse}
Given a $(n-1)$-dimensional Zinbiel algebra ${\bf Z}_{\bar 0}$. If we construct a superalgebra  ${\bf Z}$  such that ${\bf Z}_{\bar 0}{\bf Z}_{\bar 1}={\bf Z}_{\bar 1}{\bf Z}_{\bar 0}=\left\{0\right\}$  and such that ${\bf Z}_{\bar 1}{\bf Z}_{\bar 1}$ is a subspace of ${\rm Ann}_{\rm  L}\,({\bf Z}_{\bar 0})$. Then, ${\bf Z}$ is a Zinbiel superalgebra of type $(n-1, 1)$.
\end{rem}

\begin{rem}
\label{remzero}
Given a non-zero $n$-dimensional Zinbiel superalgebra ${\bf Z}$ of type $(n-1, 1)$ such that ${\bf Z}_{\bar 0}$ is the zero algebra. Then it is isomorphic to ${\bf Z}_{n,0}: f_{1}f_{1} = e_1$, simply choosing $\phi: A \to A_{n,0}$ such that $\phi(f_{1}f_{1})= e_1$. The classification of the $2$-dimensional Zinbiel superalgebras follows by this statement.

%
%
\end{rem}

We recover the classification of the $n$-dimensional Zinbiel algebras {\cite{classifications, afk22}}, as it will be required.

\begin{thm}
Given an $n$-dimensional, for $n\leq3$, non-trivial Zinbiel algebra, then it is isomorphic to only one of the following

\begin{itemize}
    \item If $n=2$, then it is
    \begin{enumerate}[$\bullet$]
        \item 
     $\mathfrak{Z}_{2, 1}: e_{1}e_{1} = e_2.$
    \end{enumerate}
    
    \item If $n=3$, then it is isomorphic to $\mathfrak{Z}_{3, 1} = \mathfrak{Z}_{2, 1} \oplus \mathbb{C}$ or to
     \begin{enumerate}[$\bullet$]
    \item  $\mathfrak{Z}_{3, 2}: e_{1}e_{1} = e_2, e_1 e_2 = \frac{1}{2}e_3, e_2 e_1 = e_3.$
    
    \item  $\mathfrak{Z}_{3, 3}: e_1 e_2 = e_3, e_2 e_1 = - e_3.$
    
     \item  $\mathfrak{Z}_{3, 4}: e_1 e_1 = e_3, e_1 e_2 = e_3,  e_2e_2 = \beta e_3.$

    \item  $\mathfrak{Z}_{3, 5}: e_1 e_1 = e_3, e_1 e_2 = e_3, e_2 e_1 = e_3.$
     \end{enumerate}

    %
    %
\end{itemize}

\end{thm}

In our classification, we will not consider non-proper superalgebras. So type $(n, 0)$, which corresponds to Zinbiel algebras, and type $(0, n)$ (zero algebra) are omitted. Also, we omit the superalgebras with $\mathfrak{Z}_{0}\mathfrak{Z}_{1} = \mathfrak{Z}_{1}\mathfrak{Z}_{0} = \mathfrak{Z}_{1}^2 = 0$, as they are split algebras.

\subsection{3-dimensional Zinbiel superalgebras}

\subsubsection{ (1, 2) superalgebras}

Let $\{ e_1, f_1, f_2\}$ be a basis of a superalgebra of ${\bf Z}={\bf Z}_{\bar{0}}\oplus {\bf Z}_{\bar{1}}$. Since ${\bf Z}_{\bar{0}}$ is the trivial one dimensional algebra, we have the following multiplication table for ${\bf Z}$: 
$$e_1f_i = a_{i}^{1}f_1 + a_{i}^{2}f_2, f_ie_1 = b_{i}^{1}f_1 + b_{i}^{2}f_2, f_if_j = c_{ij}e_1,$$
where $1\leq i,j \leq 2$. We find the equations on variables the structural constants studying case by case:

\begin{longtable}{lclcl}
    
$\mathfrak{sZ}\{e_1,e_1,f_1\}$&$=$&$0$ &$\Rightarrow $ & $(a_{1}^{1})^2+a_{1}^{1} b_{1}^{1}+a_{1}^{2} a_{2}^{1}+a_{2}^{1} b_{1}^{2}=0$ and $a_{1}^{1} a_{1}^{2}+a_{1}^{2} a_{2}^{2}+a_{1}^{2} b_{1}^{1}+a_{2}^{2} b_{1}^{2}=0$\\
    
$\mathfrak{sZ}\{e_1,e_1,f_2\}$&$=$&$0$ & $\Rightarrow $ & $a_{1}^{1} a_{2}^{1}+a_{1}^{1} b_{2}^{1}+a_{2}^{1} a_{2}^{2}+a_{2}^{1} b_{2}^{2} = 0 $ and $a_{1}^{2} a_{2}^{1}+a_{1}^{2} b_{2}^{1}+(a_{2}^{2})^2+a_{2}^{2} b_{2}^{2}=0$\\
    
$\mathfrak{sZ}\{e_1,f_1,e_1\}$&$=$&$0$ & $\Rightarrow $ & $(a_{1}^{1})^2+a_{1}^{2} a_{2}^{1}-a_{1}^{2} b_{2}^{1}+a_{2}^{1} b_{1}^{2}=0$ \\
&&&& and $a_{1}^{1} a_{1}^{2}-a_{1}^{1} b_{1}^{2}+a_{1}^{2} a_{2}^{2}+a_{1}^{2} b_{1}^{1}-a_{1}^{2} b_{2}^{2}+a_{2}^{2} b_{1}^{2}=0$\\
    
$\mathfrak{sZ}\{e_1,f_2,e_1\}$&$=$&$0$ & $\Rightarrow $ & $a_{1}^{1} a_{2}^{1}+a_{1}^{1} b_{2}^{1}+a_{2}^{1} a_{2}^{2}-a_{2}^{1} b_{1}^{1}+a_{2}^{1} b_{2}^{2}-a_{2}^{2} b_{2}^{1}= 0 $ \\&&&& and $a_{1}^{2} a_{2}^{1}+a_{1}^{2} b_{2}^{1}-a_{2}^{1} b_{1}^{2}+(a_{2}^{2})^2 = 0$\\
    
$\mathfrak{sZ}\{e_1,f_1,f_1\}$&$=$&$0$ & $\Rightarrow $ &$a_{1}^{1} c_{11}+a_{1}^{2} c_{21}=0$\\
    
$\mathfrak{sZ}\{e_1,f_1,f_2\}$&$=$&$0$ & $\Rightarrow $ & $a_{1}^{1} c_{12}+a_{1}^{2} c_{22}=0$\\
    
$\mathfrak{sZ}\{e_1,f_2,f_1\}$&$=$&$0$& $\Rightarrow $ &$a_{2}^{1} c_{11}+a_{2}^{2} c_{21}=0$\\
    
$\mathfrak{sZ}\{e_1,f_2,f_2\}$&$=$&$0$ & $\Rightarrow $& $a_{2}^{1} c_{12}+a_{2}^{2} c_{22}=0$\\
    
$\mathfrak{sZ}\{f_1,e_1,e_1\}$&$=$&$0$ & $\Rightarrow $ & $(b_{1}^{1})^2+b_{1}^{2} b_{2}^{1}=0$ and $b_{1}^{1} b_{1}^{2}+b_{1}^{2} b_{2}^{2}=0$\\
    
$\mathfrak{sZ}\{f_2,e_1,e_1\}$&$=$&$0$ & $\Rightarrow $ & $b_{1}^{1} b_{2}^{1}+b_{2}^{1} b_{2}^{2}=0$ and $b_{1}^{2} b_{2}^{1}+(b_{2}^{2})^2=0$\\
    
$\mathfrak{sZ}\{f_1,e_1,f_1\}$&$=$&$0$ & $\Rightarrow $ & $a_{1}^{1} c_{11}+a_{1}^{2} c_{12}+b_{1}^{2} c_{12}-b_{1}^{2} c_{21}=0$\\
    
$\mathfrak{sZ}\{f_1,e_1,f_2\}$&$=$&$0$ & $\Rightarrow $ & $a_{2}^{1} c_{11}+a_{2}^{2} c_{12}-b_{1}^{1} c_{12}-b_{1}^{2} c_{22}+b_{2}^{1} c_{11}+b_{2}^{2} c_{12}=0$\\
    
$\mathfrak{sZ}\{f_2,e_1,f_1\}$&$=$&$0$ & $\Rightarrow $& $a_{1}^{1} c_{21}+a_{1}^{2} c_{22}+b_{1}^{1} c_{21}+b_{1}^{2} c_{22}-b_{2}^{1} c_{11}-b_{2}^{2} c_{21}=0$\\
    
$\mathfrak{sZ}\{f_2,e_1,f_2\}$&$=$&$0$ & $\Rightarrow $ & $a_{2}^{1} c_{21}+a_{2}^{2} c_{22}-b_{2}^{1} c_{12}+b_{2}^{1} c_{21}=0$\\
    
$\mathfrak{sZ}\{f_1,f_1,e_1\}$&$=$&$0$ & $\Rightarrow $ & $a_{1}^{1} c_{11}+a_{1}^{2} c_{12}+b_{1}^{1} c_{11}+b_{1}^{2} c_{12}=0$\\
    
$\mathfrak{sZ}\{f_1,f_2,e_1\}$&$=$&$0$ & $\Rightarrow $ & $a_{2}^{1} c_{11}+a_{2}^{2} c_{12}+b_{2}^{1} c_{11}+b_{2}^{2} c_{12}=0$\\
    
$\mathfrak{sZ}\{f_2,f_1,e_1\}$&$=$&$0$ & $\Rightarrow $ & $a_{1}^{1} c_{21}+a_{1}^{2} c_{22}+b_{1}^{1} c_{21}+b_{1}^{2} c_{22}=0$\\
    
$\mathfrak{sZ}\{f_2,f_2,e_1\}$&$=$&$0$ & $\Rightarrow $ & $a_{2}^{1} c_{21}+a_{2}^{2} c_{22}+b_{2}^{1} c_{21}+b_{2}^{2} c_{22}=0$\\

$\mathfrak{sZ}\{f_1,f_1,f_1\}$&$=$&$0$ & $\Rightarrow $ & $a_{1}^{1} c_{11}=0$ and $a_{1}^{2} c_{11}=0$\\
    
 $\mathfrak{sZ}\{f_1,f_1,f_2\}$&$=$&$0$ & $\Rightarrow $ & $a_{2}^{1} c_{11}-b_{1}^{1} c_{12}+b_{1}^{1} c_{21} = 0 $ and $a_{2}^{2} c_{11}-b_{1}^{2} c_{12}+b_{1}^{2} c_{21}=0$\\
    
  $\mathfrak{sZ}\{f_1,f_2,f_1\}$&$=$&$0$ & $\Rightarrow $ & $a_{1}^{1} c_{12}+b_{1}^{1} c_{12}-b_{1}^{1} c_{21}=0$ and $a_{1}^{2} c_{12}+b_{1}^{2} c_{12}-b_{1}^{2} c_{21}= 0$\\
    
$\mathfrak{sZ}\{f_1,f_2,f_2\}$&$=$&$0$ & $\Rightarrow $ & $a_{2}^{1} c_{12} = 0 $ and $a_{2}^{2} c_{12}=0$\\
    
$\mathfrak{sZ}\{f_2,f_1,f_1\}$&$=$&$0$ & $\Rightarrow $ & $a_{1}^{1} c_{21} = 0$ and $a_{1}^{2} c_{21}=0$\\
    
$\mathfrak{sZ}\{f_2,f_1,f_2\}$&$=$&$0$ & $\Rightarrow $ & $a_{2}^{1} c_{21}-b_{2}^{1} c_{12}+b_{2}^{1} c_{21} = 0 $ and $a_{2}^{2} c_{21}-b_{2}^{2} c_{12}+b_{2}^{2} c_{21}=0$\\
    
 $\mathfrak{sZ}\{f_2,f_2,f_1\}$&$=$&$0$ & $\Rightarrow $ & $a_{1}^{1} c_{22}+b_{2}^{1} c_{12}-b_{2}^{1} c_{21} = 0 $ and $a_{1}^{2} c_{22}+b_{2}^{2} c_{12}-b_{2}^{2} c_{21}=0$\\
    
 $\mathfrak{sZ}\{f_2,f_2,f_2\}$&$=$&$0$ & $\Rightarrow $ & $a_{2}^{1} c_{22}=0$ and $a_{2}^{2} c_{22}=0$.
\end{longtable}

By solving the system of equations, we have one of the following possibilities for ${\bf Z}$: 

\begin{enumerate}[(a)]
    \item 
  $ 
    \begin{tabular}{lcllcllcllcl}
    $f_1 f_1$&$ =$&$ \lambda_{11}e_1, 
    $&$f_1 f_2 $&$=$&$\lambda_{12}e_1,$ &
    $f_2 f_1 $&$=$&$\lambda_{21}e_1,$&
    $f_2 f_2$&$ =$&$\lambda_{22}e_1.$
    \end{tabular}
       $

\item $ 
    \begin{tabular}{lcllcllcllcl}
    $f_1e_1 $&$=$&$ \mu f_1 -\frac{\lambda_{11}}{\lambda_{12}}\mu f_2 , $  &
    $f_2e_1 $&$=$&$ \frac{\lambda_{12}}{\lambda_{11}}\mu f_1 - \mu f_2,$&
    $ f_1 f_1 $&$= $&$\lambda_{11}e_1, $ \\
    $f_1 f_2$&$ =$&$\lambda_{12}e_1,$&
    $ f_2 f_1$&$=$&$ \lambda_{12}e_1,$&
    $ f_2 f_2$&$ = $&$\frac{\lambda_{12}^2}{\lambda_{11}}e_1.$
    \end{tabular}
     $  

\item $ 
    \begin{tabular}{lcllcllcllcl}
    $ e_1 f_1 $&$=$&$\mu f_{1} - \frac{\mu^2}{\mu'} f_{2},$ &
    $e_1 f_2 $&$= $&$\mu' f_{1} - \mu f_{2}.$
    \end{tabular}
      $

\item $ 
    \begin{tabular}{lcllcllcllcl}
    $e_1 f_1 $&$= $&$\mu f_1 - \frac{\mu \nu}{\nu'}f_2,$&\\
    $ e_1 f_2 $&$=$&$ \frac{\mu \nu'}{\nu}f_1 - \mu f_2,$ &
    $f_1 e_1 $&$=$&$ \nu f_1 - \frac{\nu^2}{\nu'}f_2,$&$ f_2 e_1$&$ = $&$\nu' f_1 - \nu f_2.$\\
    \end{tabular}$  

\item   \begin{tabular}{lcllcllcllcl}$f_1e_1 $&$=$&$ \mu f_2, $&$f_1 f_1 $&$=$&$ \mu' e_1.$\end{tabular} 
    
\item   \begin{tabular}{lcllcllcllcl}$f_2e_1 $&$=$&$ \mu f_1,$&$ f_2 f_2 $&$=$&$ \mu'e_1.$\end{tabular}
    
\item   \begin{tabular}{lcllcllcllcl}$e_1f_1 $&$= $&$\mu f_2,$&$ f_1e_1 $&$=$&$ \mu'f_2.$\end{tabular}

\item   \begin{tabular}{lcllcllcllcl}$e_1f_2 $&$= $&$\mu f_1,$&$ f_2e_1 $&$=$&$ \mu'f_1.$\end{tabular}
    
\end{enumerate}

\begin{lem}
 The isomorphism clases of the eight cases above can be reduced to the study of the cases (a), (e) and (g).
\begin{enumerate}
    \item The superalgebras (e) and (f) are isomorphic. 
    
    \item The superalgebras (g) and (h) are isomorphic. 
    
    \item The superalgebras (b) are isomorphic to (e) for $\mu'=\frac{\lambda_{12}^2}{\lambda_{11}}$.
    
    \item The superalgebras (c) are isomorphic to (g) for $\mu'=0$.
    
    \item The superalgebras (d) are isomorphic to (g) for $\mu'=\nu$.
\end{enumerate}
\end{lem}
\begin{proof}

The statements (1) and (2) are trivial. The statements (3), (4) and (5) can be shown, respectively, using the following maps.
\begin{longtable}{lcllcllcl}
$\phi_{1}(e_1)$&$= $&$e_1,$&$ \phi_{1}(e_2)$&$=$&$ \frac{\lambda_{11}}{\lambda_{12}}e_2, $&$\phi_{1}(e_3)$&$=$&$ e_2 + e_3,$\\
$\phi_{2}(e_1)$&$=$&$ e_1,$&$ \phi_{2}(e_2)$&$=$&$ \frac{\mu}{\mu'}e_2 + \frac{\mu}{\mu'} e_3,$&$ \phi_{2}(e_3)$&$=$&$ e_2,$\\
$\phi_{3}(e_1)$&$=$&$ e_1, $&$\phi_{3}(e_2)$&$=$&$ e_2, $&$\phi_{3}(e_3)$&$= $&$\frac{\nu'}{\nu}e_2 - \frac{\nu'}{\nu}e_3.$
\end{longtable}

\end{proof}

Next, we study the remaining cases.

{\bf Case (a)} is equivalent to one of the following, depending on the parameter.
\begin{itemize}
    \item If $\lambda_{11}\neq0$ and $\lambda_{12}\neq\lambda_{21}$. Now, choose \begin{center}
        $\phi(e_1)=\frac{\lambda_{11}}{(\lambda_{12}-\lambda_{21})^2} e_1$, $\phi(e_2)=\frac{\lambda_{11}}{\lambda_{21}-\lambda_{12}}e_2$ and $\phi(e_3)=  \frac{\lambda_{12}}{\lambda_{21}-\lambda_{12}}e_2 + e_3$,
    \end{center} then
    $${\bf z}_{3, 1}^{\alpha}: 
    f_1 f_1 = e_1,  \quad   f_2 f_1 = e_1,  \quad f_2 f_2 =\alpha e_1.
   $$
    
    \item If $\lambda_{11}\neq0$ and $\lambda_{12}=\lambda_{21}$. Then, by choosing the map \begin{center} $\phi(e_1)= \lambda_{11}^{-1}e_1,$ $\phi(e_2)= e_3$, $\phi(e_3)= -\frac{\sqrt{\lambda_{11}\lambda_{22}-\lambda_{12}^2}}{\lambda_{11}}e_2 + \lambda_{12}\lambda_{11}^{-1}e_3$,\end{center} we obtain
    $${\bf z}_{3, 2}:f_1 f_1 =e_1,\quad  f_2 f_2 =e_1.$$
Note that, if $\lambda_{22}={\lambda_{12}^2}{\lambda_{11}^{-1}}$, then this is equivalent to some case in (e).
    
    \item If $\lambda_{11}=0$, $\lambda_{12}\neq\lambda_{21}$ and ($\lambda_{12}\neq -\lambda_{21}$ or $\lambda_{22}\neq0$). Choose the map \begin{center}
    $\phi(e_1)=\frac{\lambda_{22}}{(\lambda_{12}-\lambda_{21})^2}e_1$, $\phi(e_2)=\frac{\lambda_{21}}{\lambda_{12}-\lambda_{21}}e_2 + e_3$ and $\phi(e_3)=\frac{\lambda_{22}}{\lambda_{12}-\lambda_{21}}$,
    \end{center} we obtain ${\bf z}_{3, 1}^{\alpha}$.

    \item If $\lambda_{11}=0$ and $\lambda_{12}=\lambda_{21}\neq0$, then we have ${\bf z}_{3, 2}$.
    
    \item If $\lambda_{11}=\lambda_{12}=\lambda_{21}=0$, then we obtain
    $${\bf z}_{3, 3}:f_1 f_1 =e_1,$$
    choosing $\phi(e_1)= \lambda_{22}^{-1}e_1$, $\phi(e_2)=e_3$ and $\phi(e_3)=e_2$.
    
    \item If $\lambda_{11}=\lambda_{22}=0$ and $\lambda_{12}=-\lambda_{21}$. Choose the map 
    \begin{center} $\phi(e_1)=\lambda_{12}^{-1}e_1$, $\phi(e_2)=e_2$ and $\phi(e_3)=e_3$ \end{center}  to obtain
    $${\bf z}_{3, 4}:f_1 f_2 = e_1, \quad f_2 f_1 =-e_1.$$
    
\end{itemize}

{\bf Case (e)} is equivalent to one of the following, depending on the parameters.
\begin{itemize}
    \item If $\mu\neq0$ and $\mu'\neq0$, then the map given by \begin{center}$\phi(e_1)=\mu'^{-1}e_1$, $\phi(e_2)=e_2$ and $\phi(e_3)=\mu^{-1}\mu'^{-1}e_3$\end{center}  shows that it is isomorphic to
    $${\bf z}_{3, 5}:f_1e_1 = f_2,\quad f_1 f_1 = e_1.$$
    
    \item If $\mu=0$ and $\mu'\neq0$, then, choosing $\phi(e_1)=\mu'^{-1}e_1$, $\phi(e_2)=e_2$ and $\phi(e_3)=e_3$, we have  ${\bf z}_{3, 3}.$

    \item If $\mu\neq0$ and $\mu'=0$, choose $\phi(e_1)=e_1$, $\phi(e_2)=e_2$ and $\phi(e_3)=\mu^{-1}e_3$ to obtain
    $${\bf z}_{3, 6}:f_1e_1 = f_2.$$
\end{itemize}

{\bf Case (g)} is equivalent to one of the following, depending on the parameters.
\begin{itemize}
    \item If $\mu'\neq0$, then, with $\phi$ such that $\phi(e_1)=e_1$, $\phi(e_2)=e_2$ and $\phi(e_3)=\mu'^{-1}e_3$, we have
    $${\bf z}_{3, 7}:e_1f_1 = \alpha f_2, f_1e_1 = f_2.$$
    
    \item If $\mu'=0$, choosing $\phi(e_1)=e_1$, $\phi(e_2)=e_2$ and $\phi(e_3)=\mu^{-1}e_3$ then we obtain
    $${\bf z}_{3, 8}:e_1f_1 = f_2.$$
    
\end{itemize}

\subsubsection{ (2, 1) superalgebras}

\begin{itemize}
    \item Even part $\mathfrak{Z}_{2, 0}$. Then, by Remark \ref{remzero}, we have ${\bf z}_{3, 0}$.
    
    \item Even part $\mathfrak{Z}_{2, 1}$. By Lemma \ref{lem(n-1,1)} and Remark \ref{remconverse}, we have that every superalgebra constructed on $\mathfrak{Z}_{2, 1}$ is of the form
    $$e_{1}e_{1} = e_2, \quad
    f_{1}f_{1} =  \lambda_{2} e_2. $$

    Choose the linear map $\phi$ such that $\phi(e_1)=\lambda_{2}^{-\frac{1}{2}} e_1,\phi(e_2)=\lambda_{2}^{-1}e_2,\phi(e_3)=e_3,$
    to obtain the superalgebra    
    $${\bf z}_{3, 9}:    e_{1}e_{1} = e_2,   f_{1}f_{1} = e_2. $$

\end{itemize}

Summing up, we have the classification of the $3$-dimensional Zinbiel superalgebras.

\begin{thm}\label{clasi3}

Given a $3$-dimensional complex non-split Zinbiel superalgebra ${\bf Z}$, then it is isomorphic to a $3$-dimensional Zinbiel algebra or to only one of the following.

\begin{longtable}{lclcllcllcl}

 ${\bf z}_{3, 1}^{\alpha}$&$:$&$ f_1 f_1 $&$ =$&$e_1$ &$f_2 f_1 $&$=$&$ e_1$ &    $f_2 f_2$&$ =$&$\alpha e_1$\\

${\bf z}_{3, 2}$&$:$&$f_1 f_1 $&$=$&$e_1$&$ f_2 f_2$&$ =$&$e_1$\\

${\bf z}_{3, 3}$&$:$&$f_1 f_1 $&$=$&$e_1$\\

${\bf z}_{3, 4}$&$:$&$f_1 f_2 $&$=$&$ e_1$  &    $f_2 f_1 $&$=$&$-e_1$\\

${\bf z}_{3, 5}$&$:$&$f_1e_1 $&$=$&$ f_2$&$ f_1 f_1 $&$=$&$ e_1$\\

${\bf z}_{3, 6}$&$:$&$f_1e_1 $&$=$&$ f_2$\\

${\bf z}_{3, 7}$&$:$&$  e_1f_1 $&$=$&$ \alpha f_2$  & $f_1e_1 $&$=$&$ f_2$\\

 ${\bf z}_{3, 8}$&$:$&$e_1f_1 $&$=$&$ f_2$\\

${\bf z}_{3, 9}$&$:$&$e_{1}e_{1} $&$=$&$ e_2$ &    $f_{1}f_{1} $&$=$&$ e_2.$
    \end{longtable}

\end{thm}

\section{Finite-dimensional Zinbiel superalgebras are nilpotent}

It is well-known that finite-dimensional Zinbiel algebras are nilpotent over an arbitrary field \cite{zinb5} (also see \cite{nilpotency} for context).
It is a natural question to wonder if this is also true in the case of Zinbiel superalgebras.
Note, for instance, that all the 3-dimensional Zinbiel superalgebras are nilpotent (Theorem \ref{clasi3}). It turns out that the answer is positive, as we will see in this section.

\begin{defn}

Given an algebra ${\bf Z}$ we define the right annihilator of an element $a\in {\bf Z}$ as the set
$$RC(a) = \left\{ x\in {\bf Z}: ax=0\right\}.$$

\end{defn}

\begin{lem}\label{rightcenter}
Given a right-commutative superalgebra ${\bf Z}$, then for homogeneous elements $a_1, a_2\in {\bf Z}$, 
we have $RC(a_1)\subseteq RC(a_1a_2)$.
\end{lem}
\begin{proof}
Given $x\in RC(a_1)$ and suppose $x=x_0 + x_1$, for $x_{i} \in {\bf Z}_{\bar{i}}$. Then, since $a_1$ is homogeneous $a_1x= a_1x_0 + a_1 x_1 = 0$ implies $a_1x_0=0$ and $a_1 x_1=0$. Hence, we have
\begin{center}
    $(a_1a_2)x = (a_1a_2)x_0 + (a_1a_2)x_1 = (a_1x_0)a_2 + (-1)^{|a_2|}(a_1x_1)a_2 = 0.$\end{center}
    
\end{proof}

The first key lemma of this section is the following.

\begin{lem} \label{lem1nil}
Given a finite-dimensional Zinbiel superalgebra ${\bf Z}$, there exists a homogeneous element $e$ such that $e{\bf Z}=0$.    
\end{lem}
\begin{proof}
 Since ${\bf Z}_0$ is a Zinbiel algebra, it is nilpotent. Assume it has nilpotency index $N$, then we have some non-zero element $ e_0\in {\bf Z}_{\bar{0}}^{N-1}$ such that $e_0{\bf Z}_{\bar{0}}={\bf Z}_{\bar{0}}e_0=0$. Construct $e$ as follows. 

\begin{enumerate}
    \item Fix $e=e_0$.
    
    \item If there is some $x\in {\bf Z}_{\bar{1}}$ such that $ex\neq0$, set $e_0=ex$. Then $x\in RC(e_0)$, by the Zinbiel superidentity. Otherwise, set $e=e_0$ and finish the iteration.
    
    \item Repeat from (1).
\end{enumerate}

Note that the element obtained in each iteration is homogeneous, so by Lemma \ref{rightcenter}, in each iteration, the right annihilator becomes bigger. Also, since the algebra is finite-dimensional, this process is finite, as it is enough to run it for a basis of ${\bf Z}_{\bar{1}}$. So we conclude $RC(e)={\bf Z}$, that is $e{\bf Z}=0$.
\end{proof}

\begin{lem}\label{ideal1}
Given $I$ a right ideal of a Zinbiel superalgebra ${\bf Z}$, then ${\bf Z}I$ is an ideal.
\end{lem}
\begin{proof}
We have ${\bf Z}({\bf Z}I) \subseteq {\bf Z}^2I + {\bf Z}(I{\bf Z}) \subseteq {\bf Z}I$ and $({\bf Z}I){\bf Z}\subseteq {\bf Z}^2I\subseteq {\bf Z}I$. 
\end{proof}

The next result follows by the previous lemmas.

\begin{lem}
Any Zinbiel superalgebra of dimension $n>1$ has a proper graded ideal.    
\end{lem}
\begin{proof}
Given a finite-dimensional Zinbiel superalgebra ${\bf Z}$, by Lemma \ref{lem1nil}, there exists  an element $e\in {\bf Z}_{\bar{i}}$, for $i=0$ or $i=1$, such that $e{\bf Z}=0$.  
Now, if ${\bf Z}e = 0$, then the vector space generated by $E$ is a proper graded ideal. Conversely, if ${\bf Z}e \neq 0$, choose $I={\bf Z}e$, then since the linear spam of $e$ is a right ideal, $I$ is an ideal, by Lemma \ref{ideal1}.

To show that it is a proper ideal, we have to prove that its dimension is lower than $n$. Choose a basis $e_1, e_2, \ldots, e_{n}$ of ${\bf Z}$ such that $e_1=E$, then the ideal $I$ is linearly generated by the elements $e_1 E=0, e_2 E, \ldots, e_{n}E$, therefore, at most it has dimension $n-1$.

The ideal $I$ is graded as a consequence of $e$ being homogeneous.
\end{proof}

Now, we can prove the first main result of this section.

\begin{lem}\label{solvable}
    Any finite-dimensional Zinbiel superalgebra is solvable.
\end{lem}

\begin{proof}
Let  ${\bf Z}$ be a finite-dimensional Zinbiel superalgebra. We proceed by induction on the dimension $n$. If $n=1$, we have the trivial one-dimensional algebra, which is solvable. Now, if $n>1$ and the statement is true for up to dimension $n-1$, then $A$ has a proper graded ideal $I$, so $I$ and ${\bf Z}/I$ are Zinbiel superalgebras of dimension lower than $n$, therefore they are solvable. Hence, ${\bf Z}$ itself is solvable.
\end{proof}

\begin{prop}\label{anni1}
    Let $I$ be a minimal ideal of a finite-dimensional Zinbiel superalgebra ${\bf Z}$. Let $J$ be a minimal right ideal of ${\bf Z}$ such that $J\subseteq I$. Then $I=J$.
\end{prop}
\begin{proof}
 The proof is identical to the proof of \cite[Proposition 2.2]{zinb5}.
\end{proof}

\begin{cor}\label{anni2}
    Let $I$ be a minimal ideal of a finite-dimensional Zinbiel superalgebra ${\bf Z}$, then we have $I{\bf Z}={\bf Z}I=0$. Hence, we have $\dim I = 1$.
\end{cor}

\begin{proof}
 The proof is identical to the proof of \cite[Corollary 2.3]{zinb5}.
\end{proof}

Observe that the previous result implies that any finite-dimensional Zinbiel superalgebra has a non-trivial annihilator. However, this is not enough to prove that any finite-dimensional Zinbiel superalgebra is nilpotent, we need the next straightforward remark.

\begin{rem}
Let $I$ be a minimal ideal of a Zinbiel superalgebra ${\bf Z}$. Suppose it is generated by some element $e=e_0+e_1 \in {\bf Z}$, where $e_i\in {\bf Z}_i$. Note that $e\in {\rm Ann}({\bf Z})$. Then $e{\bf Z}_{\bar{i}}=0$ implies $e_0 {\bf Z}_{\bar{i}} = 0 $ and $e_1 {\bf Z}_{\bar{i}} = 0$ (resp. ${\bf Z}_{\bar{i}}e=0$ implies ${\bf Z}_{\bar{i}} e_0 = 0 $ and ${\bf Z}_{\bar{i}} e_1 = 0$), and we have $e_i\in {\rm Ann}(A)$. Moreover, $I_i=\langle e_i\rangle$ is an ideal. Furthermore, if $e_i\neq0$, then $I_i$ is a proper graded ideal. Hence, we have the following corollary.

\end{rem}

 \begin{cor}\label{anni11}
Any finite-dimensional Zinbiel superalgebra ${\bf Z}$ has a minimal ideal which is graded. Moreover, there exist a homonegeous element $e\in {\bf Z}$ such that $e\in {\rm Ann}({\bf Z})$.
 \end{cor}

Finally, Corollary \ref{anni11} enables us to prove the main result of this section. 

\begin{thm}
    Any finite-dimensional Zinbiel superalgebra is nilpotent.
\end{thm}

\begin{proof}
Let  ${\bf Z}$ be a finite-dimensional Zinbiel superalgebra. We proceed by induction on the dimension $n$. 
If $n=1$, we have the trivial one-dimensional algebra, which is nilpotent. Now, suppose we have that any finite-dimensional Zinbiel superalgebra is nilpotent up to dimension $n-1$ for $n>1$. Since ${\bf Z}$ has a graded ideal $I$ of dimension one (generated by a homogeneous element) such that ${\bf Z}I=I{\bf Z}=0$, then ${\bf Z}/I$ is a Zinbiel superalgebra of dimension $n-1$, and so it is nilpotent. Therefore, ${\bf Z}$ is nilpotent.
\end{proof}





\bibliographystyle{amsplain}

\end{document}